\documentclass{amsart} 
\usepackage{amsmath}
\usepackage[mathscr]{eucal} 
\usepackage{amssymb}
\usepackage{latexsym}
\usepackage{amsthm} 
\theoremstyle{plain}
\newtheorem{theorem}{Theorem}[section]

\newtheorem{theorema}{Theorem A\!\!}

\newtheorem{proposition}[theorem]{Proposition}

\newtheorem{lemma}[theorem]{Lemma}  
\newcommand{\supp}{\mathop{\mathrm{supp}}\nolimits}

\numberwithin{equation}{section}  
\theoremstyle{definition}

\theoremstyle{remark}
\newtheorem{remark}[theorem]{Remark}

\def\XXint#1#2#3{{\setbox0=\hbox{$#1{#2#3}{\int}$}
\vcenter{\hbox{$#2#3$}}\kern-.5\wd0}}

\title[ Weak type estimates for functions of Marcinkiewicz type]
{Weak type estimates for functions of Marcinkiewicz type with fractional 
integrals of mixed homogeneity } 
\author{Shuichi Sato} 
 
\begin{document} 
\address{Department of Mathematics,
Faculty of Education, Kanazawa University, Kanazawa 920-1192, Japan}
\email{shuichi@kenroku.kanazawa-u.ac.jp}
\begin{abstract} 
We prove the endpoint weak type estimate for square functions of 
Marcinkiewicz type 
with fractional integrals associated with non-isotropic dilations.  
This generalizes a result of C. Fefferman on 
functions of Marcinkiewicz type by considering fractional 
integrals of mixed homogeneity in place of the Riesz potentials of  
Euclidean structure.  
\end{abstract}
  \thanks{
  \endgraf
  }
\thanks{The author is partly supported
by Grant-in-Aid for Scientific Research (C) No. 16K05195, Japan 
Society for the  Promotion of Science.}

\maketitle  

\section{Introduction}  

Let $P=\mathop{\mathrm{diag}}(a_1, \dots, a_n)$  be an 
$n\times n$ real diagonal matrix such that $a_j\geq 1$, $1\leq j\leq n$. 
 Define a dilation group 
$\{A_t\}_{t>0}$ on $\Bbb R^n$ by 
$A_t=\mathop{\mathrm{diag}}(t^{a_1}, \dots, t^{a_n})$. 
We see that $|A_tx|$ is strictly increasing as a function of $t$ on 
$\Bbb R_+=(0,\infty)$ for $x\neq 0$, where $|x|$ denotes the Euclidean norm.  
Define a norm function $\rho(x)$, $x\neq 0$, to be the unique positive real 
number $t$ such that $|A_{t^{-1}}x|=1$ and let $\rho(0)=0$.  
Then $\rho(A_tx)=t\rho(x)$, $t>0$, $x\in \Bbb R^n$, 
 and the following properties of $\rho(x)$ and $A_t$ are known 
(see \cite{CT, Ca, FR}): 
\begin{enumerate} 
\item[(A)]  $\rho\in C^\infty(\Bbb R^n\setminus \{0\})$; 
\item[(B)]  $\rho(x+y)\leq \rho(x)+\rho(y)$; 
\item[(C)]  $\rho(x)\leq 1$ if and only if $|x|\leq 1$; 
\item[(D)]   $|x|\leq \rho(x)$ if $|x|\leq 1$; 
\item[(E)]   $|x|\geq \rho(x)$ if $|x|\geq 1$; 
\item[(F)] 
we have a polar coordinates expression for the Lebesgue measure: 
$$\int_{\Bbb R^n}f(x)\,dx=\int_0^\infty\int_{S^{n-1}} f(A_t\theta)
t^{\gamma-1}\mu(\theta)\, d\sigma(\theta)\,dt, \quad 
\gamma=\text{{\rm trace} $P$},  $$  
where $\mu$ is a strictly positive $C^\infty$ function on the unit sphere 
$S^{n-1}=\{|x|=1\}$ and 
$d\sigma$ is the Lebesgue surface measure on $S^{n-1}$. 
\end{enumerate} 
\par 
Define a Riesz potential operator by 
\begin{equation}\label{rieszp}
\widehat{I_\alpha(f)}(\xi)=(2\pi \rho(\xi))^{-\alpha}\hat{f}(\xi) 
\end{equation} 
for $0<\alpha<\gamma$, where the Fourier transform $\hat{f}$ is defined as 
$$ 
\hat{f}(\xi)=\int_{\Bbb R^n} f(x)e^{-2\pi i\langle x,\xi
\rangle}\, dx, \quad \langle x,\xi\rangle=\sum_{j=1}^n x_j\xi_j,  
$$    
with  $x=(x_1, \dots, x_n)$, $\xi=(\xi_1, \dots, \xi_n)$ 
(see also Remark \ref{re8.1} in Section 8 for the definition of 
$I_\alpha$). 
Let $\|f\|_p$ denote the $L^p$ norm of a function $f$ in $L^p(\Bbb R^n)$. 
Let $\mathscr S(\Bbb R^n)$ be the Schwartz class of rapidly decreasing smooth 
functions on $\Bbb R^n$. 
Then the following result is known (see \cite[Theorem 4.1]{CT2}). 
\begin{theorema}\label{TA}   
Let $1<p<\infty$, $0<\alpha<\gamma/p$, $1/p-1/q=\alpha/\gamma$.  
Suppose that $f$ is in $\mathscr S(\Bbb R^n)$ and $\supp(\hat{f})$ does not 
contain the origin.  
Then 
$$\|I_\alpha(f)\|_q\leq C\|f\|_p.  $$  
\end{theorema}  
Define 
\begin{equation*}
D_\alpha(f)(x)=\left(\int_{\Bbb R^n}|I_\alpha(f)(x+y)- I_\alpha(f)(x)|^2 
\rho(y)^{-\gamma-2\alpha}\, dy\right)^{1/2}.    
\end{equation*} 
In this note we shall prove the following. 
\begin{theorem}\label{T1.1}
Let $0<\alpha<1$ and $p_0=2\gamma/(\gamma+2\alpha)$.  
Suppose that $p_0>1$.  Then 
\begin{enumerate}
\item[$(1)$]  
the operator $D_\alpha$ is bounded on $L^p(\Bbb R^n)$ if $p_0<p <\infty;$ 
\item[$(2)$]  $D_\alpha$ is of weak type $(p_0,p_0):$  
\begin{equation}\label{weak}  
\sup_{\beta>0}\beta^{p_0}\left|\left\{x\in \Bbb R^n: D_\alpha(f)(x)>\beta 
\right\}\right| \leq C\|f\|_{p_0}^{p_0}, 
\end{equation} 
where $|E|$ denotes the Lebesgue measure of a set $E$. 
\end{enumerate}  
\end{theorem}  
We note that $p_0>1$ for all $\alpha\in (0,1)$ if $n\geq 2$.  
See Remark \ref{re8.2} in Section 8 for the optimality of Theorem \ref{T1.1}. 
When $A_tx=tx$ and $\rho(x)=|x|$, part (1) is due to \cite{St} and part (2) is 
stated in \cite{F}, a proof of which can be found in \cite{ChW}.  
The proof of 
\cite{ChW} uses properties of harmonic functions by extending $I_\alpha(f)$ 
as a harmonic function on the upper half space 
$\Bbb R^{n+1}_+=\Bbb R^n \times (0, \infty)$  and  results are  
stated in weighted settings.  
Also, see \cite{SeW} for results related to part 
(1) with 
$A_t=\mathop{\mathrm{diag}}(t,\dots, t, t^2)$, $\gamma=n+1$, $n\geq 2$.  
\par 
In 1938, a square function, now called the Marcinkiewicz function, 
was introduced  by \cite{M} in the setting of periodic functions on 
$\Bbb R^1$, which can be used to investigate differentiability of functions 
and characterize function spaces including Sobolev spaces.  
A generalization of the Marcinkiewicz 
function to higher dimensions can be found in \cite{St},  where also 
$D_\alpha(f)$, a variant of 
the Marcinkiewicz function, is considered when $\rho(x)=|x|$. 
We refer to 
\cite{AMV}, \cite{HL}, \cite{Sa}, \cite{Sa4}, \cite{Sa2}, \cite{Sa3} and 
\cite{SWYY}  
 for relevant, recent results on the relations between 
functions of Marcinkiewicz type and Sobolev spaces. 
\par 
To prove part (1) for $p\in (p_0, 2]$, 
 we first prove  $L^2$ boundedness of $D_\alpha$ by 
applying the Fourier transform  and the result for $p\in (p_0, 2)$   
follows from the Marcinkiewicz interpolation theorem between the $L^2$ 
boundedness and the weak type boundedness of part (2). 
\par 
The proof of part (2) we give in this note is motivated by the proof of 
the weak type estimate for the Littlewood-Paley function $g_\lambda^*$ in 
\cite{F}. The proof of \cite{F} uses some properties of the Poisson kernel 
$$P(x,t)=c_n \frac{t}{(|x|^2+t^2)^{(n+1)/2}}, 
\quad c_n=\frac{\Gamma((n+1)/2)}{\pi^{(n+1)/2}}, $$ 
 associated with harmonic functions on the upper half space $\Bbb R^{n+1}_+$ 
 (see \cite[Chapter I]{SW}). 
One of them is related to the  formula  
\begin{equation}\label{semi}
\int_0^\infty P_t*f(x) t^{\alpha-1}\, dt= \Gamma(\alpha )I_\alpha(f)(x), 
\end{equation}  
where $\widehat{I_\alpha f}(\xi)=(2\pi |\xi|)^{-\alpha}\hat{f}(\xi)$, 
$P_t(x)=t^{-n}P(x/t)=P(x,t)$, with $P(x)=P(x,1)$. Also, some regularities 
on $P(x,t)$ are used, although  properties of harmonic functions,   
like that applied in \cite{ChW} to prove the 
special case of Theorem \ref{T1.1} (2) mentioned above, 
are not used in an essential way. 
In proving Theorem \ref{T1.1} (2),  we are able to successfully 
generalize the methods of 
\cite{F} for the estimate of $g_\lambda^*$  to the present 
situation,  where results from differential equations, 
like  harmonicity,  are not readily available.     
Our proof of Theorem \ref{T1.1} (2) in this note 
is new even in the case of the Euclidean norm setting. 
\par 
 To prove part (2) of Theorem \ref{T1.1}, we consider 
 the function $K$  defined by 
\begin{equation} \label{parapo}
K(x)=\int_{\Bbb R^n} e^{-2\pi\rho(\xi)} e^{2\pi i \langle x, \xi\rangle}
\, d\xi, 
\end{equation} 
as a substitute for the Poisson kernel $P(x)$ and consider the function 
$K_t*f(x)$, where $K_t(x)=t^{-\gamma}K(A_t^{-1}x)$.  Then we have 
an analogue of \eqref{semi} for the general $I_\alpha(f)$ in \eqref{rieszp} 
(see \eqref{seminon} below). 
Also, we have some results analogous to the regularities for $P$ 
 (see Lemma \ref{L3.1} below).  
We shall apply these results to estimate the bad part 
arising from the Calder\'{o}n-Zygmund decomposition derived from 
the Whitney type decomposition of open 
sets in homogeneous spaces (see \cite{CW, CW2}). To treat the good part 
we shall apply the $L^2$ boundedness of $D_\alpha$. 
\par 
 In Section 2, we shall state  
the Calder\'{o}n-Zygmund decomposition of $f\in L^p(\Bbb R^n)$, $1<p<\infty$, 
 at height $\beta^p$, $\beta>0$, needed for the proof of 
 Theorem \ref{T1.1} (2). 
 Some properties of functions related to $K$ in \eqref{parapo} will be 
 shown in Section 3. 
\par 
We shall prove the $L^2$ boundedness of $D_\alpha$ in Section 4. 
Part (2) of Theorem \ref{T1.1} will be shown by applying the 
$L^2$ boundedness and the Calder\'{o}n-Zygmund decomposition 
in Sections 4 through 6. 
We shall show part (1) of Theorem \ref{T1.1} for $p>2$ 
in Section 7  by proving 
weighted $L^2$ estimates for $D_\alpha$ with $A_1$-weights. 
Finally, we shall have some concluding remarks in Section 8.

\section{Decomposition results}  
For $x\in \Bbb R^n$ and $r>0$, let $B(x,r)$ be the ball centered at $x$ with 
radius $r$ defined by $\rho$: $B(x,r)=\{y\in \Bbb R^n: \rho(x-y)<r \}$.  
Then we have the following (see \cite{CW2} and also \cite{CW}). 
\begin{lemma} \label{L2.1}
Let $O$ be an open bounded set in $\Bbb R^n$ and $N\geq 1$.   
Then 
There exists a sequence  $\{ B(c_j,r_j)\}_{j=1}^\infty$ of balls for which 
we have 
\begin{enumerate} 
\item $O=\cup_{j=1}^\infty B(c_j,r_j);$ 
\item there exists $C>0$ such that $\sum_{j=1}^\infty 
\chi_{B(c_j,Nr_j)}\leq C$, where $\chi_E$ denotes the characteristic function 
of a set $E;$   
\item there exists $C_1\geq 1$ such that 
$B(c_j,C_1Nr_j)\cap (\Bbb R^n\setminus  O)\neq \emptyset$.  
\end{enumerate}
\end{lemma}  
Applying this we can prove the next result (see \cite{CW}).
\begin{lemma} \label{L2.2} 
Let $\beta>0$, $f\in L^p$, $1\leq p<\infty$. Suppose that $f$ is 
compactly supported. 
Let $N\geq 1$. Then there exists 
a sequence $\{ B(c_j,r_j)\}_{j=1}^\infty$ of balls such that  
\begin{enumerate} 
\item $\sum_j\chi_{B(c_j,Nr_j)}\leq C;$  
\item $|\Omega|\leq C\beta^{-p}\|f\|_p^p$, where $\Omega=\cup B(c_j,r_j);$ 
\item $|f(x)|\leq C\beta$ if $x\in \Bbb R^n\setminus \Omega;$  
\item $|B(c_j,r_j)|^{-1}\int_{B(c_j,r_j)}|f(x)|^p \, dx\leq C\beta^p$.  
\end{enumerate}
\end{lemma} 
\begin{proof}  
Define the Hardy-Littlewood maximal function
$$M(f)(x)=\sup_{x\in B} \frac{1}{|B|}\int_B |f(y)|\, dy,$$ 
where the supremum is taken over all the balls $B$ which contain $x$.  
Let 
$$\Omega=\{x\in \Bbb R^n: M(|f|^p)>\beta^p \}. $$  
Then $\Omega$ is open and bounded. 
Clearly, we have part (3). 
By Lemma \ref{L2.1} with $\Omega$ in place of  $O$, we have a sequence 
$\{ B(c_j,r_j)\}_{j=1}^\infty$ of balls as in 
Lemma \ref{L2.1}. So we have part (1). 
Also, part (2) holds true since it is known that $M$ is of weak type $(1,1)$.  
\par 
By part (3) of Lemma \ref{L2.1}, there exist $h=C_1N\geq 1$ and $y\in 
\Bbb R^n\setminus  \Omega$ such that $y\in B(c_j,hr_j)$. Thus 
$$\frac{1}{|B(c_j,r_j)|}\int_{B(c_j,r_j)}|f(x)|^p \, dx 
\leq h^\gamma M(|f|^p)(y)
\leq h^\gamma \beta^p, $$ 
which implies part (4). 
 
\end{proof} 
Lemma \ref{L2.2} is used to prove the following (see \cite{CW}). 
\begin{lemma} \label{L2.3} 
Let $\beta$, $f\in L^p$, $p$, $N$ and $\{ B(c_j,r_j)\}_{j=1}^\infty$ 
be as in Lemma $\ref{L2.2}$.  
Then there exist a bounded function $g$ and a sequence $\{ b_j\}_{j=1}^\infty$ 
of functions in $L^p$ such that 
\begin{enumerate} 
\item $f=g+\sum_{j=1}^\infty b_j;$  
\item $|g(x)|\leq C\beta;$ 
\item $\|g\|_p\leq C\|f\|_p;$  
\item $b_j(x)=0$ if $x\in B(c_j,r_j)^c$ for all $j$, where $E^c$ denotes the 
complement of a set $E;$    
\item $\int b_j(x) \, dx=0$ for all $j;$  
\item $\|b_j\|_p^p\leq C\beta^p|B(c_j,r_j)|$  for all $j;$   
\item $\sum_{j=1}^\infty |B(c_j,r_j)|\leq C\beta^{-p}\|f\|_p^p$.    
\end{enumerate}
\end{lemma} 
\begin{proof} 
Define a function $h_j$ on $\Bbb R^n$ by 
$$h_j(x)=\frac{\chi_{B(c_j,r_j)}(x)}{\sum_{j=1}^\infty \chi_{B(c_j,r_j)}(x)} 
\quad \text{if $x \in \Omega$,} 
$$ 
and $h_j(x)=0$ if  $x\in \Omega^c$, where 
$\Omega=\cup_{j=1}^\infty B(c_j,r_j)$.  
Let 
$$g(x)= \sum_{j=1}^\infty \left(\frac{1}{|B(c_j,r_j)|} 
\int_{B(c_j,r_j)}f(y) h_j(y) \, dy\right)\chi_{B(c_j,r_j)}(x)+
f(x)\chi_{\Omega^c}(x) $$  
and 
$$ b_j(x)=f(x)h_j(x) -\left(\frac{1}{|B(c_j,r_j)|} 
\int_{B(c_j,r_j)}f(y) h_j(y) \, dy\right)\chi_{B(c_j,r_j)}(x). 
$$ 
Then by the definitions and Lemma \ref{L2.2} we easily have the 
assertions (1) through (6). Also, since $\{ B(c_j,r_j)\}_{j=1}^\infty$ is 
finitely overlapping,  by part (2) of Lemma \ref{L2.2} we have part (7). 
This completes the proof. 
\end{proof}

\section{Some estimates for Fourier transforms}  

In this section we prove some estimates for the Fourier transform 
of the function $e^{-2\pi t\rho(\xi)}$ and 
its derivatives  needed in proving Theorem 
\ref{T1.1}.  

\begin{lemma}\label{L3.1} 
We have the following estimates$:$   
\begin{gather}  
  \label{ker1}
|K(x)|\leq C(1+\rho(x))^{-\gamma-1},  \quad \text{where $K$ is as in 
\eqref{parapo}, }
\\ 
 \label{ker2}
\left|Q(x)\right|\leq C(1+\rho(x))^{-\gamma-1},  
\\ 
\intertext{where}  
Q(x)=
-\int_{\Bbb R^n} 
2\pi\rho(\xi)e^{-2\pi\rho(\xi)} e^{2\pi i \langle x, \xi\rangle}
\, d\xi,                                               \notag 
\\ 
\left|\int_{\Bbb R^n} \xi_k e^{-2\pi\rho(\xi)} 
e^{2\pi i \langle x, \xi\rangle}
\, d\xi\right|\leq C(1+\rho(x))^{-\gamma-1-a_k}, \quad 1\leq k\leq n,   
\label{ker3}
\\ 
\left|\int_{\Bbb R^n} \xi_k\rho(\xi)e^{-2\pi\rho(\xi)} 
e^{2\pi i \langle x, \xi\rangle}
\, d\xi\right|\leq C(1+\rho(x))^{-\gamma-1-a_k}, \quad 1\leq k\leq n,   
\label{ker4}
\\ 
\left|\int_{\Bbb R^n} \xi_k\xi_le^{-2\pi\rho(\xi)} 
e^{2\pi i \langle x, \xi\rangle}
\, d\xi\right|\leq C(1+\rho(x))^{-\gamma-1-a_k-a_l}, \quad 1\leq k, l\leq n.  
\label{ker5}
\end{gather}
\end{lemma} 
To prove this lemma we need the following two estimates for the derivatives of 
functions involving homogeneous functions (Lemmas \ref{L3.2}, \ref{L3.3}). 
\begin{lemma}\label{L3.2}
  Let $b=(b_1, \dots, b_n)$ be a multi-index of non-negative integers $b_j$, 
  $1\leq j\leq n$.  Let $H$ be 
homogeneous of degree $m \in \Bbb R$ with respect to the dilation $A_t$ and in 
$C^\infty(\Bbb R^n\setminus\{0\})$. 
Then we have 
$$|\partial^b H(\xi)|\leq C_b\rho(\xi)^{m-\langle a,b\rangle}, \quad 
\xi\in \Bbb R^n\setminus\{0\}, $$  
where $\partial^b=\partial_1^{b_1}\dots \partial_n^{b_n}$ with 
$\partial_j=\partial/\partial \xi_j$ and $a=(a_1, \dots, a_n)$. 
\end{lemma}
\begin{proof} Differentiating in $\xi'$ both sides of the equality
$$t^m H(\xi')=H(t^{a_1}\xi_1', \dots, t^{a_n}\xi_n'), \quad t>0,$$  
which follows from the homogeneity, and putting $t=\rho(\xi)$, 
$\xi'=A_{\rho(\xi)}^{-1}\xi\in S^{n-1}$, we get the estimates as claimed. 
\end{proof} 
\begin{lemma}\label{L3.3} Choose $\varphi \in C_0^\infty$  such that 
$\varphi(\xi)=1$ if $\rho(\xi)\leq 1/2$ and $\supp(\varphi)\subset B(0,1)$.   
 Let $F\in C^\infty(\Bbb R)$.  Let $\epsilon_k=0$ or $1$, $1\leq k\leq n$, 
 $m\in \Bbb R$. Then we have 
$$\left|\partial^b(\xi_k^{\epsilon_k}\xi_l^{\epsilon_l}\rho(\xi)^mF(\rho(\xi))
\varphi(\xi) )\right|\leq 
C_b\rho(\xi)^{m+\epsilon_ka_k+\epsilon_l a_l-\langle a, b \rangle}.  $$ 
\end{lemma}
\begin{proof} Applying Lemma \ref{L3.2} with $H=\rho$,  
we observe that 
\begin{equation}\label{e1.7+} 
\left|\partial^b(F(\rho(\xi))
\varphi(\xi) )\right|\leq 
C_b\rho(\xi)^{1-\langle a, b \rangle}, \quad b\neq 0.  
\end{equation} 
By  Leibniz's formula we have 
\begin{equation*}
\partial^b(\xi_k^{\epsilon_k}\xi_l^{\epsilon_l}\rho(\xi)^mF(\rho(\xi))
\varphi(\xi) )= \sum_{b'+b''=b}C_{b',b''}\left[\partial^{b'}(\xi_k^{\epsilon_k}
\xi_l^{\epsilon_l}\rho(\xi)^m)\right]\left[ \partial^{b''}(F(\rho(\xi))
\varphi(\xi) )\right]. 
\end{equation*}  
Since $S(\xi)=\xi_k^{\epsilon_k}\xi_l^{\epsilon_l}\rho(\xi)^m$ 
is homogeneous of 
degree $\epsilon_k a_k+\epsilon_l a_l +m$, by Lemma \ref{L3.2} and 
\eqref{e1.7+} we see that 
\begin{align*} 
&\left|\partial^b(S(\xi)F(\rho(\xi))
\varphi(\xi) )\right|
\leq \sum_{b'+b''=b}C_{b',b''}\left|\partial^{b'}S(\xi)\right| 
\left|\partial^{b''}(F(\rho(\xi))\varphi(\xi))\right|
\\ 
&\leq C\rho(\xi)^{m+\epsilon_ka_k+\epsilon_l a_l-\langle a, b \rangle}+ 
C\sum_{b'+b''=b, b''\neq 0}\rho(\xi)^{m+\epsilon_ka_k+\epsilon_l a_l-
\langle a, b' \rangle}\rho(\xi)^{1-\langle a, b'' \rangle}, 
\end{align*}
for $\xi\in B(0,1)\setminus \{0\}$.   
This completes the proof. 
\end{proof}  
Applying Lemma \ref{L3.3} and integration by parts, 
we can prove the following estimate, from 
which Lemma \ref{L3.1} readily follows. 
\begin{lemma} \label{L3.4} 
Let $G(\xi)=\xi_k^{\epsilon_k}\xi_l^{\epsilon_l}\rho(\xi)^mF(\rho(\xi))
\varphi(\xi)$ be as in Lemma $\ref{L3.3}$.  We assume that $m>-\gamma$. 
 Then 
\begin{equation*}
\left|\int_{\Bbb R^n} G(\xi) e^{2\pi i \langle x, \xi\rangle}
\, d\xi\right|\leq C(1+\rho(x))^{-\gamma-m-\epsilon_ka_k-\epsilon_l a_l}. 
\end{equation*} 
\end{lemma}
\begin{proof} 
Let $\Phi\in C_0^\infty(\Bbb R_+)$ be such that 
$\supp(\Phi)\subset \{1/2\leq r\leq 2\}$, $\Phi\geq 0$ and 
$\sum_{j=0}^\infty\Phi(2^j\rho(\xi))=1$ on $B(0,1)\setminus \{0\}$.  
Decompose 
\begin{equation*} 
\int_{\Bbb R^n} G(\xi) e^{2\pi i \langle x, \xi\rangle}\, d\xi=
\sum_{j=0}^\infty2^{-j\gamma}\int_{\Bbb R^n} G(A_{2^{-j}}\xi)\Phi(\rho(\xi)) 
e^{2\pi i \langle A_{2^{-j}}x, \xi\rangle}\, d\xi.  
\end{equation*} 
We write $x=A_{\rho(x)}x'=(\rho(x)^{a_1}x_1', \dots, \rho(x)^{a_n}x_n')$ with  
$x'\in S^{n-1}$. We may assume that $|x_1'|=\max_{1\leq j\leq n}|x_j'|$ 
without  loss of generality. 
Then applying integration by parts 
\begin{multline*} 
 2^{-j\gamma}\left|\int_{\Bbb R^n} G(A_{2^{-j}}\xi)\Phi(\rho(\xi)) 
e^{2\pi i \langle A_{2^{-j}}x, \xi\rangle}\, d\xi\right|
\\ 
\leq C2^{-j\gamma}\sum_{h=h'+h''}\int 2^{-ja_1 h'}
\left|(\partial_1^{h'}G)(A_{2^{-j}}\xi)\right|
\left|\partial_1^{h''}\Phi(\rho(\xi))\right|(2^{-ja_1}\rho(x)^{a_1})^{-h}
\, d\xi.  
\end{multline*} 
By Lemma \ref{L3.3} with $b=(h',0, \dots, 0)$ we have 
\begin{align*} 
& 2^{-j\gamma}\left|\int_{\Bbb R^n} G(A_{2^{-j}}\xi)\Phi(\rho(\xi)) 
e^{2\pi i \langle A_{2^{-j}}x, \xi\rangle}\, d\xi\right|
\\ 
&\leq C2^{-j\gamma}\sum_{h=h'+h''}\int_{1/2\leq\rho(\xi)\leq 2} 2^{-ja_1 h'} 
2^{-j(m+\epsilon_ka_k+\epsilon_l a_l-h'a_1)}
2^{jha_1}\rho(x)^{-ha_1}\, d\xi 
\\
&\leq C  2^{j(ha_1-m-\epsilon_ka_k-\epsilon_l a_l-\gamma)}\rho(x)^{-ha_1}. 
\end{align*} 
Thus if $ha_1-m-\epsilon_ka_k-\epsilon_l a_l-\gamma>0$, $\rho(x)>1$, 
\begin{multline}\label{e1.8}
\sum_{0\leq j\leq \log_2\rho(x)}2^{-j\gamma}\left|\int_{\Bbb R^n}
 G(A_{2^{-j}}\xi)\Phi(\rho(\xi)) 
e^{2\pi i \langle A_{2^{-j}}x, \xi\rangle}\, d\xi\right| 
\\ 
\leq C\rho(x)^{-ha_1}
\sum_{0\leq j\leq \log_2\rho(x)}
2^{j(ha_1-m-\epsilon_ka_k-\epsilon_l a_l-\gamma)}
\leq C\rho(x)^{-m-\epsilon_ka_k-\epsilon_l a_l-\gamma}. 
\end{multline} 
\par 
Also, by Lemma \ref{L3.3} with $b=0$ we see that 
\begin{multline}\label{e1.9}
\sum_{j> \log_2\rho(x)}2^{-j\gamma}\left|\int_{\Bbb R^n} G(A_{2^{-j}}\xi)
\Phi(\rho(\xi)) 
e^{2\pi i \langle A_{2^{-j}}x, \xi\rangle}\, d\xi\right| 
\\ 
\leq C\sum_{j> \log_2\rho(x)}2^{j(-m-\epsilon_ka_k-\epsilon_l a_l-\gamma)}
\leq C\rho(x)^{-m-\epsilon_ka_k-\epsilon_l a_l-\gamma}. 
\end{multline} 
Combining \eqref{e1.8} and \eqref{e1.9}, we get the desired result, since 
the estimate for $\rho(x)\leq 1$ is obvious. 
\end{proof} 

\begin{proof}[Proof of Lemma $\ref{L3.1}$]
Let $\varphi$ be as in Lemma \ref{L3.3}. 
Decompose 
\begin{equation*}  
 e^{-2\pi\rho(\xi)}= -2\pi\rho(\xi)A(-2\pi\rho(\xi))\varphi(\xi)+\varphi(\xi)+
e^{-2\pi\rho(\xi)}(1-\varphi(\xi)),  
\end{equation*} 
where $A(s)=(e^s-1)/s$.  To prove \eqref{ker1}, it suffice to show that  
\begin{equation*}
\left|\int_{\Bbb R^n} \rho(\xi)A(-2\pi\rho(\xi))\varphi(\xi) 
e^{2\pi i \langle x, \xi\rangle}
\, d\xi\right|\leq C(1+\rho(x))^{-\gamma-1}, 
\end{equation*} 
which follows from Lemma \ref{L3.4} with $m=1, \epsilon_k=0, \epsilon_l=0$. 
The other estimates can be shown similarly by applying Lemma \ref{L3.4} 
suitably. 
\end{proof}

\section{Outline of Proof of Theorem $\ref{T1.1}$ for $p\in [p_0, 2]$}  

We first prove $L^2$ boundedness of $D_\alpha$ for $0<\alpha<1$. 
By the Plancherel theorem we have 
\begin{align*} 
\|D_\alpha(f)\|_2^2 
&=\int_{\Bbb R^n}\rho(y)^{-\gamma-2\alpha}\left(\int_{\Bbb R^n}
|I_\alpha(f)(x+y)- I_\alpha(f)(x)|^2 \, dx\right)\, dy 
\\ 
&=\int_{\Bbb R^n}\rho(y)^{-\gamma-2\alpha}\left(\int_{\Bbb R^n}
\left|(2\pi\rho(\xi))^{-\alpha}\hat{f}(\xi)
\left(e^{2\pi i\langle y, \xi\rangle}-1\right)\right|^2 \, d\xi\right)
\, dy  
\\ 
&= (2\pi)^{-2\alpha}\int_{\Bbb R^n}|\hat{f}(\xi)|^2 \left(\int_{\Bbb R^n}
\left|e^{2\pi i\langle y, \xi'\rangle}-1\right|^2
\rho(y)^{-\gamma-2\alpha} \, dy\right)\, d\xi,  
\end{align*} 
where $\xi'=A_{\rho(\xi)}^{-1}\xi$.  By (D) of Section 1 we have 
$$\int_{\rho(y)\leq 1}
\left|e^{2\pi i\langle y, \xi'\rangle}-1\right|^2 
\rho(y)^{-\gamma-2\alpha}\, dy\leq \int_{\rho(y)\leq 1}
4\pi^2 \rho(y)^2\rho(y)^{-\gamma-2\alpha}\, dy <\infty $$
since $\alpha<1$. Also, we have 
$$\int_{\rho(y)\geq 1} 
\left|e^{2\pi i\langle y, \xi'\rangle}-1\right|^2 
\rho(y)^{-\gamma-2\alpha}\, dy\leq 
\int_{\rho(y)\geq 1} 
4\rho(y)^{-\gamma-2\alpha}\, dy<\infty $$  
for $\alpha>0$.   
Combining results,  we see that 
$$ \|D_\alpha(f)\|_2^2 \leq C\|\hat{f}\|_2^2=C\|f\|_2^2, $$ 
which proves the $L^2$ boundedness. 
\par   
To prove \eqref{weak} we may assume that $f$ is bounded and compactly 
supported. 
Let $\beta>0$, $p_0=2\gamma/(\gamma+2\alpha)$.  
We apply Lemmas \ref{L2.2} and \ref{L2.3} with these $f$, $\beta$ and with 
$N=2$, $p=p_0$.  Then we have the sequence 
$\{ B(c_j,r_j)\}_{j=1}^\infty$ of  balls  of Lemma \ref{L2.2} and 
the decomposition 
$f=g+b$, $b=\sum_{j=1}^\infty b_j$,  of Lemma \ref{L2.3}. 
It suffices 
to prove  
\begin{gather} \label{e4.1}
|\{x\in \Bbb R^n: D_\alpha(g)(x)>\beta \}|\leq C\beta^{-p_0}\|f\|_{p_0}^{p_0} 
\\ 
\intertext{and}
|\{x\in \Bbb R^n: D_\alpha(b)(x)>\beta \}|\leq 
C\beta^{-p_0}\|f\|_{p_0}^{p_0}.  
\label{e4.2}
\end{gather}  
The estimate \eqref{e4.1} easily follows from the $L^2$ boundedness 
of $D_\alpha$ as follows. By Chebyshev's inequality along with (2) and (3) 
of Lemma \ref{L2.3}, since $1<p_0<2$,  we have  
\begin{multline*}
|\{x\in \Bbb R^n: D_\alpha(g)(x)>\beta \}|
\\ 
\leq \beta^{-2}\|D_\alpha(g)\|_2^2 
\leq C\beta^{-2}\|g\|_2^2  \leq C\beta^{-p_0}\|g\|_{p_0}^{p_0}
\leq C\beta^{-p_0}\|f\|_{p_0}^{p_0}.  
\end{multline*} 
\par 
It remains to prove \eqref{e4.2}.   
Let $K$ be as in \eqref{parapo} and  
\begin{equation*}
v(x,t)=K_t*b(x), \quad V(x,t)=K_t*I_\alpha(b)(x).   
\end{equation*} 
Then   
\begin{equation}\label{seminon}
V(x,t)=\frac{1}{\Gamma(\alpha)} \int_0^\infty v(x,t+s)s^{\alpha-1}\, ds.   
\end{equation} 
We have 
\begin{multline}\label{e4.3}
|I_\alpha(b)(x+y)- I_\alpha(b)(x)| 
\leq \left|V(x+y,\rho(y)) -I_\alpha(b)(x+y)\right| 
\\ 
+ \left|V(x,\rho(y)) -I_\alpha(b)(x)\right| 
+|V(x+y,\rho(y))-V(x,\rho(y))|.     
\end{multline} 
\par 
Let 
\begin{align*} 
J^{(1)}(x)&=\Gamma(\alpha)^{2}
\int_{\Bbb R^n}\left|V(y,\rho(y-x)) -I_\alpha(b)(y)\right|^2 
\rho(y-x)^{-\gamma-2\alpha}\,dy, 
\\ 
J^{(2)}(x)&=\Gamma(\alpha)^{2}
\int_{\Bbb R^n}\left|V(x,\rho(y-x)) -I_\alpha(b)(x)\right|^2 
\rho(y-x)^{-\gamma-2\alpha}\,dy, 
\\ 
J^{(3)}(x)&=\Gamma(\alpha)^{2}
\int_{\Bbb R^n}\left|V(y,\rho(y-x))-V(x,\rho(y-x))\right|^2 
\rho(y-x)^{-\gamma-2\alpha}\,dy.  
\end{align*} 
By \eqref{seminon} we can rewrite 
\begin{align*} 
J^{(1)}(x)&=
\int_{\Bbb R^n}\left|\int_0^\infty\, dt\int_0^{\rho(y-x)} 
\partial_0 v(y,s+t)
t^{\alpha-1}\, ds \right|^2 \rho(y-x)^{-\gamma-2\alpha}\,dy, 
\\ 
J^{(2)}(x)&=
\int_{\Bbb R^n}\left|\int_0^{\infty}\, dt\int_0^{\rho(y-x)} 
\partial_0 v(x,s+t)
t^{\alpha-1}\, ds \right|^2 \rho(y-x)^{-\gamma-2\alpha}\,dy,   
\end{align*} 
where $\partial_0=\partial/\partial s$. 
Let $\Omega_1=\cup_j B(c_j,2r_j)$, $\Omega_2=\cup_j B(c_j,4r_j)$. 
Since $|\Omega_2|\leq C\beta^{-p_0}\|f\|_{p_0}^{p_0}$,  
by \eqref{e4.3}, the estimate \eqref{e4.2} follows 
from the inequalities  
\begin{gather}
\label{e4.5} 
\int_{\Omega_2^c} J^{(1)}(x)\, dx \leq C\beta^{2-p_0}\|f\|_{p_0}^{p_0}, 
\\ 
\label{e4.6} 
\int_{\Omega_2^c} J^{(2)}(x)\, dx \leq C\beta^{2-p_0}\|f\|_{p_0}^{p_0}, 
\\ 
\label{e4.7} 
\int_{\Omega_2^c} J^{(3)}(x)\, dx \leq C\beta^{2-p_0}\|f\|_{p_0}^{p_0}.  
\end{gather} 
This will prove part (2) of Theorem \ref{T1.1}.  As mentioned in Section 1, 
part (1) for $p\in(p_0,2)$ 
follows by the Marcinkiewicz interpolation theorem between 
the estimate in part (2) and the $L^2$ boundedness.  
This will complete the proof of Theorem \ref{T1.1} for $p\in [p_0, 2]$. 
\par 
 We shall prove \eqref{e4.5}, \eqref{e4.6} in Section 5 and \eqref{e4.7} in 
Section 6.  
 
 \section{Proofs of the estimates \eqref{e4.5} and \eqref{e4.6}}

We first prove the estimate \eqref{e4.5}.     
Let $\chi(r)=\chi_{(0.1]}(r)$. Then 
\begin{equation*} 
J^{(1)}(x)=\int_{\Bbb R^n}\left|\int_0^\infty\int_0^\infty 
\chi\left(\frac{s}{\rho(y-x)}\right)\chi\left(\frac{s}{t}\right)
|t-s|^{\alpha-1}\partial_0 v(y,t)\, dt \, ds\right|^2 
\rho(y-x)^{-\gamma-2\alpha}\,dy.  
\end{equation*} 
We have 
\begin{equation} \label{formulaw}   
\int_0^{t\wedge \rho(y-x)} (t-s)^{\alpha-1}\, ds=W_\alpha(t, \rho(y-x)), 
\end{equation}
where \begin{equation*}   
W_\alpha(t, s)=
\frac{1}{\alpha}\left(t^\alpha-(t-(t\wedge s))^\alpha\right), 
\quad t, s\geq 0, 
\end{equation*}
with $a\wedge b=\min(a,b)$.  We note that $W_\alpha\geq 0$ and 
$\int_0^\infty W_\alpha(t, 1)\, dt/t<\infty$ when $0<\alpha<1$.  
Using \eqref{formulaw}, we write 
\begin{equation*} 
J^{(1)}(x)=\int_{\Bbb R^n}\left|\int_0^\infty
W_\alpha(t, \rho(y-x)) \partial_0 v(y,t)\, dt \right|^2 
\rho(y-x)^{-\gamma-2\alpha}\,dy.  
\end{equation*} 
Define $v_j(y,t)=b_j*K_t(y)$ and $B_j=B(c_j,r_j)$, 
$\widetilde{B}_j=B(c_j, 2r_j)$.   
 Let 
\begin{gather*} 
J^{(1)}_1(x)=\int_{\Bbb R^n}\left|\int_0^\infty
W_\alpha(t, \rho(y-x))
\sum_{y\in \widetilde{B}_j^c}\partial_0 v_j(y,t)\, dt \right|^2 
\rho(y-x)^{-\gamma-2\alpha}\,dy, 
\\ 
\intertext{where $\sum_{y\in \widetilde{B}_j^c} $ means that the summation 
is over all $j$ such that $y\in \widetilde{B}_j^c$; similar notation will be 
used in what follows, and let}  
J^{(1)}_2(x)=\Gamma(\alpha)^{2}\int_{\Bbb R^n}\left|
\sum_{y\in \widetilde{B}_j}\left(
K_{\rho(y-x)}*I_\alpha(b_j)(y)- I_\alpha(b_j)(y)\right) \right|^2 
\rho(y-x)^{-\gamma-2\alpha}\,dy.  
\end{gather*} 
\par 
To estimate $J^{(1)}_1$, we show that 
\begin{equation*}
\left|\sum_{y\in \widetilde{B}_j^c}\partial_0 v_j(y,t) \right|\leq C\beta/t.   
\end{equation*}  
This can be seen as follows.  Let $E(x)=(1+\rho(x))^{-\gamma-1}$. 
Note that if $y\in \widetilde{B}_j^c $ 
\begin{equation*}
\sup_{z\in B_j}E_t(y-z)\leq C\inf_{z\in B_j}E_t(y-z). 
\end{equation*} 
Therefore, by Lemma \ref{L2.2} (1), Lemma \ref{L2.3} and \eqref{ker2},  
\begin{align*}
\left|\sum_{y\in \widetilde{B}_j^c}\partial_0 v_j(y,t)
\right|&\leq Ct^{-1}\sum_{y\in \widetilde{B}_j^c}
\sup_{z\in B_j}E_t(y-z)\int|b_j(z)|\, dz 
\\ 
&\leq Ct^{-1}\sum_{y\in \widetilde{B}_j^c}(\sup_{z\in B_j}E_t(y-z))\beta|B_j| 
\\ 
&\leq Ct^{-1}\beta\sum_{y\in \widetilde{B}_j^c}\int \chi_{B_j}(z)E_t(y-z) \, dz
\\ 
&\leq Ct^{-1}\beta\|E\|_1.  
\end{align*}  
Thus 
\begin{align*} 
\left|\int_0^\infty
W_\alpha(t, \rho(y-x))
\sum_{y\in \widetilde{B}_j^c}\partial_0 v_j(y,t)\, dt \right| 
&\leq C\beta \int_0^\infty
W_\alpha(t, \rho(y-x))\,  \frac{dt}{t} 
\\ 
&= C\beta \rho(y-x)^\alpha \int_0^\infty 
W_\alpha(t, 1)\, \frac{dt}{t},  
\end{align*}
and hence 
\begin{equation*} 
\int_{\Bbb R^n} J^{(1)}_1(x)\, dx 
\leq C\beta\int_{\Bbb R^n}\int_0^\infty
\left(\int_{\Bbb R^n}W_\alpha(t, \rho(x))
\rho(x)^{-\gamma-\alpha}\, dx\right)
\left|\sum_{y\in \widetilde{B}_j^c}\partial_0 v_j(y,t)\right| \, dy\,dt.  
\end{equation*} 
\par 
We note that 
\begin{equation*}
\int_{\Bbb R^n} W_\alpha(t, \rho(x))
\rho(x)^{-\gamma-\alpha}\, dx  
= \int_{\Bbb R^n} W_\alpha(1, \rho(x))\rho(x)^{-\gamma-\alpha}\, dx <\infty,  
\end{equation*}
if $0<\alpha<1$.  
This implies that 
\begin{equation*} 
\int_{\Bbb R^n} J^{(1)}_1(x)\, dx  \leq C\beta\iint_{\Bbb R^{n+1}_+ } 
\left|\sum_{y\in \widetilde{B}_j^c}\partial_0 v_j(y,t)\right| \, dy\,dt. 
\end{equation*}  
Since $\int b_j =0$, if $y\in \widetilde{B}_j^c$, by \eqref{ker4} and 
Taylor's formula we see that  
\begin{align*} 
|\partial_0 v_j(y,t)|&= t^{-1}|Q_t*b_j(y)| =t^{-1}
\left|\int_{B_j}\left(Q_t(y-z)-Q_t(y-c_j)\right)b_j(z)\, dz\right| 
\\ 
&\leq C\sum_{k=1}^n 
t^{-1-a_k-\gamma}\left(1+t^{-1}\rho(y-c_j)\right)^{-\gamma-1-a_k}\int_{B_j}
|(z-c_j)_k||b_j(z)|\, dz 
\\ 
&\leq C\sum_{k=1}^n r_j^{a_k}t^{-1-a_k-\gamma}
\left(1+t^{-1}\rho(y-c_j)\right)^{-\gamma-1-a_k}\int_{B_j}
|b_j(z)|\, dz 
\\ 
&\leq C\beta \sum_{k=1}^n r_j^{a_k}t^{-1-a_k-\gamma} 
\left(1+t^{-1}(\rho(y-c_j)+r_j)\right)^{-\gamma-1-a_k}|B_j|,  
\end{align*} 
where the penultimate inequality follows from the estimate $|x_k|\leq C
\rho(x)^{a_k}$. 
Therefore, $\int_{\Bbb R^n} J^{(1)}_1(x)\, dx  $ is bounded by 
\begin{equation*}  
C\beta^2 \sum_j |B_j| 
\sum_{k=1}^n  \iint_{\Bbb R^{n+1}_+ }  r_j^{a_k}    
t^{-1-a_k-\gamma} 
\left(1+t^{-1}(\rho(y-c_j)+r_j)\right)^{-\gamma-1-a_k}
\, dy\,dt. 
\end{equation*}  
It is easy to see that the last integral is equal to   
\begin{equation*}
 \iint_{\Bbb R^{n+1}_+ } 
\left(1+t\right)^{-\gamma-1-a_k} (\rho(y)+1)^{-\gamma-a_k}   
\, dy\,dt.   
\end{equation*}  
Thus, by Lemma \ref{L2.3} (7) with $p=p_0$,  we have 
\begin{equation}\label{est1} 
\int_{\Bbb R^n} J^{(1)}_1(x)\, dx 
\leq C \beta^2 \sum_j |B_j|
 \leq C\beta^{2-p_0}|\|f\|_{p_0}^{p_0}. 
\end{equation}
\par 
Next, we evaluate $J^{(1)}_2$.    
By Schwarz's inequality and Lemma \ref{L2.2} (1),  we see that 
\begin{align*} 
J^{(1)}_2(x) &\leq C\int_{\Omega_1}\left|
\sum_{j} M(I_\alpha b_j)(y)\chi_{B(c_j,2r_j)}(y) \right|^2 
\rho(y-x)^{-\gamma-2\alpha}\,dy  
\\ 
&\leq C\sum_{j}\int_{B(c_j,2r_j)} 
 |M(I_\alpha b_j)(y)|^2 \rho(y-x)^{-\gamma-2\alpha}\,dy.   
\end{align*} 
  Let $x\in \Omega_2^c$.  
Then by the $L^2$ boundedness of the maximal operator $M$ we have   
\begin{align*} 
J^{(1)}_2(x) &\leq C\sum_{j}\rho(x-c_j)^{-\gamma-2\alpha}\int_{B(c_j,2r_j)} 
 |M(I_\alpha b_j)(y)|^2 \,dy   
\\ 
&\leq C\sum_{j}\rho(x-c_j)^{-\gamma-2\alpha} \|I_\alpha b_j\|_2^2.    
\end{align*} 
So, from Theorem A and Lemma \ref{L2.3} (6) with $p=p_0$ it follows that 
\begin{equation*} 
J^{(1)}_2(x) \leq C\sum_{j}\rho(x-c_j)^{-\gamma-2\alpha} \|b_j\|_{p_0}^2  
\leq C\sum_{j}\rho(x-c_j)^{-\gamma-2\alpha} \beta^2|B_j|^{2/p_0}.   
\end{equation*} 
Consequently, 
\begin{align*} 
\int_{\Omega_2^c}J^{(1)}_2(x)\, dx
&\leq C\beta^2\sum_{j}|B_j|^{2/p_0}\int_{\rho(x-c_j)\geq 4r_j}
\rho(x-c_j)^{-\gamma-2\alpha}\, dx
\\ 
&\leq C\beta^2\sum_{j}|B_j|^{2/p_0}r_j^{-2\alpha}   \notag 
\\ 
&\leq C\beta^2\sum_{j}|B_j|,                        \notag 
\end{align*}  
and hence Lemma \ref{L2.3} (7) implies 
\begin{equation} \label{est2}
\int_{\Omega_2^c}J^{(1)}_2(x)\, dx\leq C\beta^{2-p_0}\|f\|_{p_0}^{p_0}.        
\end{equation} 
The estimate \eqref{e4.5} follows from \eqref{est1} and \eqref{est2}, since 
$J^{(1)}\leq 2J^{(1)}_1 + 2J^{(1)}_2$.   
\par    
Let us prove the estimate \eqref{e4.6} next.  
In the same way as in the case of $J^{(1)}$, we can write 
\begin{equation*} 
J^{(2)}(x)=\int_{\Bbb R^n}\left|\int_0^\infty  W_\alpha(t, \rho(y-x))
\partial_0 K_t*b(x)\, dt \right|^2 
\rho(y-x)^{-\gamma-2\alpha}\,dy.  
\end{equation*} 
Interchanging the order of integration on the left hand side of \eqref{e4.6}, 
we have 
\begin{multline*} 
\int_{\Omega_2^c}J^{(2)}(x)\, dx
\\ 
= \int_{\Bbb R^n}\left(\int_{\Omega_2^c}\left|\int_0^\infty 
 W_\alpha(t, \rho(y-x))
\partial_0 v(x,t)\, dt \right|^2\rho(y-x)^{-\gamma-2\alpha}\, dx\right) \,dy. 
\end{multline*} 
For $x\in \Omega_2^c$, we note that 
$\partial_0 v(x,t)=\sum_{x\in \widetilde{B}_j^c}\partial_0v_j(x,t)$. 
Thus we can prove \eqref{e4.6} in the same way as \eqref{est1}. 

\section{Proof of the estimate \eqref{e4.7}}  
  
We note that 
\begin{equation*} 
V(x+y,\rho(y))-V(x,\rho(y))=\frac{1}{\Gamma(\alpha)} 
\int_0^\infty \left(f*K_{t+\rho(y)}(x+y)-
f*K_{t+\rho(y)}(x)\right)t^{\alpha-1}\, dt. 
\end{equation*} 
Thus we have 
\begin{multline*} 
J^{(3)}(x)
\\ 
=\int_{\Bbb R^n}\left|\int_{0}^\infty \left(b*K_{t}(y)-
b*K_{t}(x)\right)|t-\rho(y-x)|^{\alpha-1}
\chi\left(\frac{\rho(y-x)}{t}\right)\, dt\right|^2 
\rho(y-x)^{-\gamma-2\alpha} \, dy. 
\end{multline*} 
Let 
\begin{equation*}  
A_j(x,y)=\int_{0}^\infty \left(\int (K_{t}(y-z)-
K_{t}(x-z))b_j(z)\, dz\right)|t-\rho(y-x)|^{\alpha-1}
\chi\left(\frac{\rho(y-x)}{t}\right)\, dt,   
\end{equation*}
with $y\in \widetilde{B}_j^c$.  
By \eqref{ker3} and Taylor's formula, if $z\in B_j$ and $\rho(x-y)<t$ we have  
\begin{align*} 
|K_{t}(y-z)-K_{t}(x-z)|&\leq C\sum_{k=1}^n |y_k-x_k| 
t^{-a_k-\gamma}\left(1+t^{-1}\rho(y-c_j)\right)^{-\gamma-1-a_k} 
\\ 
&\leq C\sum_{k=1}^n \rho(y-x)^{a_k}
t^{-a_k-\gamma}\inf_{z\in B_j}\left(1+t^{-1}\rho(y-z)\right)^{-\gamma-1-a_k}.  
\end{align*}
Therefore if  $\rho(x-y)<t$, by Lemma \ref{L2.3} (6), 
\begin{multline*}
\left|\int (K_{t}(y-z)-K_{t}(x-z))b_j(z)\, dz\right| 
\\ 
\leq C\beta \sum_{k=1}^n \rho(y-x)^{a_k}t^{-a_k-\gamma}\int_{B_j}
 \left(1+t^{-1}\rho(y-z)\right)^{-\gamma-1-a_k}\, dz, 
\end{multline*} 
and hence Lemma \ref{L2.2} (1) implies that 
\begin{equation*}
\sum_{y\in \widetilde{B}_j^c}
\left|\int (K_{t}(y-z)-K_{t}(x-z))b_j(z)\, dz\right| 
\leq C\beta \sum_{k=1}^n C_k\rho(y-x)^{a_k}t^{-a_k}  
\end{equation*} 
with $C_k=\int_{\Bbb R^n}\left(1+\rho(z)\right)^{-\gamma-1-a_k}\, dz$. 
Consequently, 
\begin{align*} 
\sum_{y\in \widetilde{B}_j^c}|A_j(x,y)| 
&= C\beta\sum_{k=1}^n C_k\rho(y-x)^{a_k}
 \int_{\rho(y-x)}^\infty 
t^{-a_k}(t-\rho(y-x))^{\alpha-1}\, dt      \notag 
\\ 
&= C\beta\rho(y-x)^{\alpha}\sum_{k=1}^n \left(C_k
 \int_{1}^\infty t^{-a_k}(t-1)^{\alpha-1}\, dt\right).              \notag 
\end{align*}
Thus we have 
\begin{equation} \label{e1.7}  
\sum_{y\in \widetilde{B}_j^c}|A_j(x,y)|\leq C\beta\rho(y-x)^{\alpha}. 
\end{equation}  
Let $J^{(3)}_1(x)$ be    
\begin{gather*} 
\int_{\Bbb R^n}\left|\int_{0}^\infty \sum_{y\in \widetilde{B}_j^c}
\left(v_j(y,t)-v_j(x,t)\right)|t-\rho(y-x)|^{\alpha-1}
\chi\left(\frac{\rho(y-x)}{t}\right)\, dt\right|^2 
\rho(y-x)^{-\gamma-2\alpha} \, dy   
\\ 
\intertext{and let} 
J^{(3)}_2(x)=\Gamma(\alpha)^2\int_{\Bbb R^n}\left|\sum_{y\in \widetilde{B}_j}
\left(K_{\rho(x-y)}*I_\alpha(b_j)(y)-K_{\rho(x-y)}*I_\alpha(b_j)(x)\right)
\right|^2 \rho(x-y)^{-\gamma-2\alpha} \, dy. 
\end{gather*} 
 Then $J^{(3)}\leq 2J^{(3)}_1+2J^{(3)}_2$. 
From \eqref{e1.7} we see that $J^{(3)}_1(x)$ is majorized by 
\begin{equation*}   
C\beta\int_{\Bbb R^n}\left|\int_{0}^\infty \sum_{y\in \widetilde{B}_j^c}
\left(v_j(y,t)-v_j(x,t)\right)|t-\rho(y-x)|^{\alpha-1}
\chi\left(\frac{\rho(y-x)}{t}\right)\, dt\right| 
\rho(y-x)^{-\gamma-\alpha} \, dy.  
\end{equation*} 
\par 
Let 
\begin{equation*}  
R(t,y,x,z)=K_t(y-z)-K_t(x-z).   
\end{equation*}  
Then 
\begin{equation*}  
 v_j(y,t)-v_j(x,t)=\int R(t,y,x,z)b_j(z)\, dz=\int (R(t,y,x,z)-R(t,y,x,c_j))
b_j(z)\, dz.       
\end{equation*}  
By \eqref{ker5} and Taylor's formula,  we have 
\begin{multline*}  
\left|R(t,y,x,z)-R(t,y,x,c_j)\right| 
\\ 
\leq C\sum_{k=1}^n\sum_{l=1}^n \rho(z-c_j)^{a_k}\rho(x-y)^{a_l}
 t^{-\gamma-a_k-a_l}
(1+t^{-1}\rho(y-c_j))^{-\gamma-1-a_k-a_l}
\end{multline*}  
if $z\in B_j$, $\rho(x-y)<t$ and $y\in \widetilde{B}_j^c$.  Also, we note that 
$$ \int \rho(x)^{-\gamma-\alpha+a_l}|t-\rho(x)|^{\alpha-1}
\chi\left(\frac{\rho(x)}{t}\right)\, dx = t^{a_l-1}
\int_{\rho(x)<1} \rho(x)^{-\gamma-\alpha+a_l}(1-\rho(x))^{\alpha-1}\, dx. 
 $$
Using these results and Lemma \ref{L2.3} (6),  we see that 
\begin{align*} 
&\int_{\Bbb R^n} J^{(3)}_1(x)\, dx 
\\ 
&\quad \quad \leq C\beta^2 \sum_{k=1}^n\sum_{l=1}^n \iint_{\Bbb R^{n+1}_+} 
 t^{-\gamma-a_k-1}\sum_{y\in \widetilde{B}_j^c}|B_j|
r_j^{a_k}(1+t^{-1}\rho(y-c_j))^{-\gamma-1-a_k-a_l}\, dy\,dt
\\ 
&\quad \quad \leq C\beta^2 \sum_{k=1}^n\sum_{l=1}^n 
\sum_{j}|B_j| \iint_{\Bbb R^{n+1}_+} r_j^{a_k}
 t^{-\gamma-a_k-1}
(1+t^{-1}(\rho(y)+r_j))^{-\gamma-1-a_k-a_l}
\, dy\,dt.   
\end{align*}  
The last integral equals  
\begin{equation*}  
\iint_{\Bbb R^{n+1}_+}  t^{a_l}
(1+t)^{-\gamma-1-a_k-a_l}(\rho(y)+1)^{-\gamma-a_k}
\, dy\,dt.    
\end{equation*}
Thus, by Lemma \ref{L2.3} (7) with $p=p_0$ we have  
\begin{equation}\label{est3}  
\int_{\Bbb R^n} J^{(3)}_1(x)\, dx \leq C\beta^{2-p_0}\|f\|_{p_0}^{p_0}. 
\end{equation}
\par 
We next evaluate $J^{(3)}_2$.  We note that 
\begin{equation*}  
\left|K_{\rho(x-y)}*I_\alpha(b_j)(w)\right|\leq CM(I_\alpha b_j)(y), \quad 
\text{for $w=y, x$}. 
\end{equation*} 
Therefore, if $x\in \Omega_2^c$ and $p_0=2\gamma/(\gamma+2\alpha)$, 
using the Schwarz inequality and Lemma \ref{L2.2} (1), we see that 
\begin{align*} 
J^{(3)}_2(x)&\leq C\int_{\Bbb R^n}\left| \sum_j
\chi_{B(c_j,2r_j)}(y)M(I_\alpha b_j)(y)
\right|^2 \rho(y-x)^{-\gamma-2\alpha} \, dy 
\\ 
&\leq C \sum_j \int_{B(c_j,2r_j)}\left|M(I_\alpha b_j)(y)\right|^2 
\rho(y-x)^{-\gamma-2\alpha} \, dy  
\\ 
&\leq C \sum_j \rho(x-c_j)^{-\gamma-2\alpha}\int_{B(c_j,2r_j)}
\left|M(I_\alpha b_j)(y)\right|^2  \, dy.   
\end{align*} 
Consequently, the $L^2$ boundedness of the maximal function 
$M$, Theorem A and Lemma \ref{L2.3} (6) with $p=p_0$ imply
\begin{equation*}
J^{(3)}_2(x)\leq C \sum_j \rho(x-c_j)^{-\gamma-2\alpha}\|b_j\|_{p_0}^2  
\leq C \sum_j \rho(x-c_j)^{-\gamma-2\alpha}\beta^2|B_j|^{2/p_0}.    
\end{equation*}
Thus, applying Lemma \ref{L2.3} (7) with $p=p_0$, we see that 
\begin{align}\label{est4}  
\int_{\Omega_2^c} J^{(3)}_2(x)\, dx&\leq C\sum_j \beta^2|B_j|^{2/p_0}  
\int_{\rho(x-c_j)\geq 4r_j}\rho(x-c_j)^{-\gamma-2\alpha}  
\\
&\leq C\sum_j \beta^2|B_j|^{2/p_0} r_j^{-2\alpha}
\leq C\sum_j \beta^2|B_j|\leq C\beta^{2-p_0}\|f\|_{p_0}^{p_0}.   \notag 
\end{align} 
Combining \eqref{est3} and \eqref{est4}, we have \eqref{e4.7}.

\section{Proof of part $(1)$ of Theorem $\ref{T1.1}$ for $p>2$}  

Let $A_1$ be the weight class of Muckenhoupt consisting of those weights $w$ 
which satisfy $M(w)(x)\leq Cw(x)$ a.e.    Applying the methods of \cite{DR} 
we prove the following. 

\begin{proposition}\label{P7.1}
 Let $w\in A_1$. Suppose $0<\alpha<1$. Then we have 
$$\|D_\alpha(f)\|_{2,w}\leq C\|f\|_{2,w},$$  
where $\|f\|_{2,w}$ is the norm in $L^2_w$ defined as 
 $\|f\|_{2,w}=\left(\int_{\Bbb R^n}|f(x)|^2w(x)\, dx\right)^{1/2}$.  
\end{proposition}
Let 
$$\widetilde{\rho}_{m}(z)=\int 
(2\pi \rho(\xi))^{-\alpha}\widetilde{\Phi}(2^m\rho(\xi))
e^{2\pi i\langle z,\xi \rangle}\, d\xi
$$ 
with $\widetilde{\Phi}\in C_0^\infty(\Bbb R_+)$ which is identically $1$ 
on the support of $\Phi$, where $\Phi$ is as in the proof of Lemma 
\ref{L3.4}.  To prove Proposition \ref{P7.1} we need the following. 

\begin{lemma}\label{L7.2} The estimates  
\begin{gather}  \label{el1}
|\widetilde{\rho}_{m}(z)|\leq C\rho(z)^{\alpha-\gamma},    
\\ 
\label{el2}  
|\partial_s\widetilde{\rho}_{m}(z)|\leq C\rho(z)^{\alpha-\gamma-a_s},  
\quad 1\leq s\leq n,      
\end{gather} 
hold true with a positive constant $C$ independent of $m\in \Bbb Z$ 
$($the set of integers$)$.  
\end{lemma} 
\begin{proof}  
To prove \eqref{el1} we write 
$$\widetilde{\rho}_{m}(z)=2^{m(\alpha -\gamma)}\int 
(2\pi \rho(\xi))^{-\alpha}\widetilde{\Phi}(\rho(\xi))
e^{2\pi i\langle A_{2^{-m}}z,\xi \rangle}\, d\xi.  $$  
From this we easily have \eqref{el1} when $\rho(z)\leq 2^m$.  Suppose  
$2^m<\rho(z)$. Then we can prove \eqref{el1} in the same way as 
\eqref{e1.8} by applying  integration by parts.    
The estimates in \eqref{el2} can be shown similarly. 
\end{proof} 
\begin{proof}[Proof of Proposition $\ref{P7.1}$]
We may assume that $f\in \mathscr S(\Bbb R^n)$. 
We also assume that 
$\sum_{j=-\infty}^\infty\Phi(2^j\rho(\xi))=1$ for 
$\xi\in \Bbb R^n\setminus \{0\}$. 
Define the operator $\Delta_j$ by 
$$\widehat{\Delta_j(f)}(\xi) = \Phi(2^j\rho(\xi))\hat{f}(\xi) \qquad 
\text{for \quad $j \in \Bbb Z$}.$$  
Let 
\begin{multline*}
T_j(f)(x)
\\ 
=\left(\sum_{k=-\infty}^\infty \int \chi_{[1,2]}
(2^{-k} \rho(y))|I_\alpha(\Delta_{j+k}f)(x+y)- 
I_\alpha(\Delta_{j+k}f)(x)|^2 
\rho(y)^{-\gamma-2\alpha}\, dy\right)^{1/2}.    
\end{multline*} 
Then we have 
\begin{equation}\label{e7.1}  
D_\alpha(f)(x)\leq \sum_{j=-\infty}^\infty T_j(f)(x).  
\end{equation}  
If we put $S_j=\{2^{-j-1}\leq \rho(\xi)\leq 2^{-j+1}\}$,   
 the Plancherel theorem implies  
\begin{multline} \label{e7.2}  
\|T_j(f)\|_2^2  
\\ 
\leq C\sum_{k=-\infty}^\infty \int \chi_{[1,2]}(2^{-k} \rho(y))
\rho(y)^{-\gamma-2\alpha}\left(\int_{S_{j+k}}|\hat{f}(\xi)|^2
\rho(\xi)^{-2\alpha}\left|1-e^{2\pi i\langle y,\xi \rangle}\right|^2
\, d\xi\right) \, dy.  
\end{multline}  
\par 
If $2^k\leq \rho(y)\leq 2^{k+1}$, $\xi\in S_{j+k}$ and $j\geq 0$,  we see that 
\begin{equation*} 
\left|1-e^{2\pi i\langle y,\xi \rangle}\right|\leq 2\pi\sum_{l=1}^n 
|y_l\xi_l|\leq C\sum_{l=1}^n \rho(y)^{a_l}\rho(\xi)^{a_l}
\leq C\sum_{l=1}^n 2^{-ja_l}\leq C2^{-j}.  
\end{equation*}
 Also, $\left|1-e^{2\pi i\langle y,\xi \rangle}\right|\leq 2$. Therefore 
 by \eqref{e7.2} we have 
\begin{multline} \label{e7.3}  
\|T_j(f)\|_2^2  
\\ 
\leq C\sum_{k=-\infty}^\infty C2^{2j\alpha}\min(1, 2^{-2j})
\int_{S_{j+k}}|\hat{f}(\xi)|^2\, d\xi  \leq C2^{2j\alpha}\min(1, 2^{-2j})
\|f\|_2^2,  
\end{multline}  
where the last inequality follows from the bounded overlap of $\{S_j\}$ and 
the Plancherel theorem.  
 \par 
Next, when $2^k\leq \rho(y)\leq 2^{k+1}$,  we decompose 
\begin{multline} \label{e7.4}  
I_\alpha(\Delta_{j+k}f)(x+y)- 
I_\alpha(\Delta_{j+k}f)(x)
\\ 
= \int_{\rho(z)\geq 2^{k+2}} \left(\widetilde{\rho}_{j+k}(z-y)
-\widetilde{\rho}_{j+k}(z)\right)\Delta_{j+k}f(x+z)\, dz 
\\ 
+ \int_{\rho(z)< 2^{k+2}} \left(\widetilde{\rho}_{j+k}(z-y)
-\widetilde{\rho}_{j+k}(z)\right)\Delta_{j+k}f(x+z)\, dz, 
\end{multline}  
where $\widetilde{\rho}_{m}$ is as in Lemma \ref{L7.2}.  
\par 
If $2^k\leq \rho(y)\leq 2^{k+1}$ and $\rho(z)\geq 2^{k+2}$, by \eqref{el2}  
\begin{multline*}  
 \left|\widetilde{\rho}_{j+k}(z-y)
-\widetilde{\rho}_{j+k}(z)\right| 
\leq C\sum_{l=1}^n |y_l|\rho(z)^{-\gamma+\alpha-a_l}  
\leq C\sum_{l=1}^n \rho(y)^{a_l}\rho(z)^{-\gamma+\alpha-a_l}.  
 \end{multline*}  
By this and an elementary computation concerning the maximal operator 
$M$, we see that 
 the first integral on the right hand side of \eqref{e7.4} is majorized by 
\begin{equation*} 
C\sum_{l=1}^n 2^{ka_l}\int_{\rho(z)\geq 2^{k+2}}
\rho(z)^{-\gamma+\alpha-a_l}|\Delta_{j+k}f(x+z)|\, dz 
\leq C2^{k\alpha}M(\Delta_{j+k}f)(x). 
\end{equation*}
Similarly, if $2^k\leq \rho(y)\leq 2^{k+1}$, \eqref{el1}  
implies that the second integral on the right hand side of 
\eqref{e7.4} is bounded by 
\begin{multline*}  
\int_{\rho(z)\leq 2^{k+3}}
\rho(z)^{-\gamma+\alpha}|\Delta_{j+k}f(x+y+z)|\, dz + 
\int_{\rho(z)\leq 2^{k+2}}
\rho(z)^{-\gamma+\alpha}|\Delta_{j+k}f(x+z)|\, dz  
\\ 
\leq C2^{k\alpha}M(\Delta_{j+k}f)(x+y)+
 C2^{k\alpha}M(\Delta_{j+k}f)(x). 
 \end{multline*}  
Using these estimates in \eqref{e7.4}, we see that 
\begin{equation*} 
\left|I_\alpha(\Delta_{j+k}f)(x+y)- I_\alpha(\Delta_{j+k}f)(x)\right|
\leq  C2^{k\alpha}M(\Delta_{j+k}f)(x+y)+ C2^{k\alpha}M(\Delta_{j+k}f)(x) 
\end{equation*}
when $2^k\leq \rho(y)\leq 2^{k+1}$.   
\par 
Thus 
\begin{equation*} 
T_j(f)(x)^2\leq C\sum_{k=-\infty}^\infty \int \chi_{[1,2]}
(2^{-k} \rho(y))|M(\Delta_{j+k}f)(x+y)+M(\Delta_{j+k}f)(x) |^2 
\rho(y)^{-\gamma}\, dy  
\end{equation*}   
and hence, if $w\in A_1$,  
\begin{align*}  
&\int_{\Bbb R^n}T_j(f)(x)^2 w(x)\, dx 
\\ 
&\quad \quad \leq C\sum_{k=-\infty}^\infty \left(\int_{\Bbb R^n}    \notag
|M(\Delta_{j+k})f(x)|^2M(w)(x) \, dx  +  
\int_{\Bbb R^n} |M(\Delta_{j+k}f)(x) |^2w(x)\, dx\right)  
\\ 
&\quad \quad \leq C\sum_{k=-\infty}^\infty  
\int_{\Bbb R^n} |\Delta_{j+k}f(x) |^2w(x)\, dx,            \notag
\end{align*}   
where the last inequality follows from 
the defining property of the $A_1$ 
weights and the $L^2_w$ boundedness of $M$ with $w\in A_1$.   
Thus the Littlewood-Paley inequality in $L^2_w$ implies 
\begin{equation}\label{e7.5} 
\|T_j(f)\|_{2,w}\leq C\|f\|_{2,w}. 
\end{equation}
\par 
Interpolating between \eqref{e7.3} and \eqref{e7.5} with change of measures, 
and noting that for any $w\in A_1$ 
there is $\delta>0$ such that $w^{1+\delta}\in A_1$, we have 
\begin{equation*} 
\|T_j(f)\|_{2,w}\leq C2^{-\epsilon|j|}\|f\|_{2,w}
\end{equation*}   
with  some $\epsilon>0$ for $w\in A_1$, if $0<\alpha<1$.   
This implies the desired inequality in Proposition \ref{P7.1} 
via \eqref{e7.1}.  
\end{proof} 
\par 
Now we can prove part $(1)$ of Theorem $\ref{T1.1}$ for $p>2$.   
Choose a non-negative function $g$ such that $\|g\|_{(p/2)'}=1$ and 
$\|D_\alpha(f)\|_p^2=\int |D_\alpha(f)(x)|^2g(x)\, dx$, where $(p/2)'$ denotes 
the exponent conjugate to $p/2$. For $s>1$, let $M_s(g)=M(g^s)^{1/s}$.  
Then $g\leq M_s(g)$ a.e. and it is known that 
$M_s(g)\in A_1$. Thus by Proposition \ref{P7.1}
 we have 
$$\int |D_\alpha(f)(x)|^2g(x)\, dx\leq \int |D_\alpha(f)(x)|^2 M_s(g)(x)\, dx 
\leq C\int |f(x)|^2 M_s(g)(x)\, dx. $$  
Applying H\"{o}lder's inequality to the last integral with 
 $1<s<(p/2)'$, by $L^{(p/2)'}$ boundedness of $M_s$  we see that 
$$\int |D_\alpha(f)(x)|^2g(x)\, dx\leq C\|f\|_p^2\|M_s(g)\|_{(p/2)'} 
\leq C\|f\|_p^2. $$  
Combining results, we can get the desired estimate. 
\par

\section{Remarks}  

We conclude this note with three remarks.

\begin{remark}\label{re8.1} 
Let $0<\alpha<\gamma$.   
The Fourier transform of $(2\pi\rho(\xi))^{-\alpha}$ is 
a function $\mathscr R_{\alpha}(x)$ which is homogeneous of 
degree $\alpha-\gamma$ with respect to $A_t$ and in 
$C^\infty(\Bbb R^n\setminus\{0\})$
(see \cite{CT2} and \cite[Chapter I]{NS}). Thus we have 
$$I_\alpha(f)(x)=\int_{\Bbb R^n} 
\mathscr R_{\alpha}(x-z)f(z)\, dz, \quad 
f\in \mathscr S(\Bbb R^n).  $$  
\end{remark} 

\begin{remark}\label{re8.2}
Let $0<\alpha<1$, $p_0=2\gamma/(\gamma+2\alpha)$ and $p_0>1$ as in the 
hypotheses of Theorem $\ref{T1.1}$.  Then, if $1\leq p<p_0$, $D_\alpha$ is 
not of weak type $(p,p)$.   Since  $D_\alpha$ is bounded on $L^2(\Bbb R^n)$, 
by taking into account the interpolation of Marcinkiewicz, to show this it 
suffices to prove that $D_\alpha$ is not bounded on $L^p(\Bbb R^n)$ when 
$1<p<p_0$.   
\par 
To see this, 
we prove that if $D_\alpha$ is bounded on $L^p(\Bbb R^n)$ with 
$1\leq p\leq 2$, then $p\geq p_0$. 
Let $A(x)=\{ y\in \Bbb R^n: 1/2\leq \rho(y-x)\leq 1\}$. 
Let $\eta$ be a non-zero element in 
$\mathscr S(\Bbb R^n)$ with $\supp(\hat{\eta}) \subset \{1\leq \rho(\xi)
\leq 2\}$. 
Then 
\begin{align*} 
D_\alpha(\eta)(x)&\geq \left(\int_{A(0)} 
|I_\alpha(\eta)(x+y)-I_\alpha(\eta)(x)|^2\,dy\right)^{1/2}
\\ 
&\geq \left(\int_{A(0)}|I_\alpha(\eta)(x+y)|^2\, dy\right)^{1/2} 
-\left(\int_{A(0)}|I_\alpha(\eta)(x)|^2\, dy\right)^{1/2}
\\ 
&= \left(\int_{A(0)}|I_\alpha(\eta)(x+y)|^2\, dy\right)^{1/2} 
-|A(0)|^{1/2}|I_\alpha(\eta)(x)|.
\end{align*} 
Therefore  
\begin{equation}\label{e8.1}
\left(\int_{A(x)}|I_\alpha(\eta)(y)|^2\, dy\right)^{1/2} 
\leq D_\alpha(\eta)(x)+ C|I_\alpha(\eta)(x)|.  
\end{equation} 
\par 
We have 
\begin{equation}\label{e8.2}  
\left(\int_{\Bbb R^n} |I_\alpha(\eta)(y)|^2\, dy\right)^{p/2} \leq C
\int_{\Bbb R^n}\left(\int_{A(x)}|I_\alpha(\eta)(y)|^2\, dy
\right)^{p/2}\,dx 
\end{equation} 
Let $S(x,r)=\{y\in \Bbb R^n: |x-y|<r\}$ for $x\in \Bbb R^n$ and $r>0$. 
To see \eqref{e8.2}, we consider a covering of $\Bbb R^n$: 
$\cup_{j=1}^\infty A(x^{(j)})=\Bbb R^n$, for all $x^{(j)}\in S(c_j,\tau)$, 
$j=1,2, \dots$, where $S(c_j,\tau)\cap S(c_k, \tau)=\emptyset$ if $j\neq k$. 
Then we see that   
$$ 
\left(\int_{\Bbb R^n} |I_\alpha(\eta)(y)|^2\, dy\right)^{p/2} 
\leq \sum_{j=1}^\infty \left(\int_{A(x^{(j)})}|I_\alpha(\eta)(y)|^2\, dy
\right)^{p/2}  
$$ 
for all  $x^{(j)}\in S(c_j,\tau)$, $j=1,2, \dots$, since $p/2\leq 1$.  
It follows that 
\begin{align*} 
\left(\int_{\Bbb R^n} |I_\alpha(\eta)(y)|^2\, dy\right)^{p/2} 
&\leq \sum_{j=1}^\infty \inf_{x^{(j)}\in S(c_j,\tau)}
\left(\int_{A(x^{(j)})}|I_\alpha(\eta)(y)|^2\, dy\right)^{p/2} 
\\ 
&\leq C_\tau\sum_{j=1}^\infty \int_{S(c_j,\tau)}
\left(\int_{A(x)}|I_\alpha(\eta)(y)|^2\, dy\right)^{p/2}\,dx 
\\ 
&\leq   C_\tau\int_{\Bbb R^n}
\left(\int_{A(x)}|I_\alpha(\eta)(y)|^2\, dy\right)^{p/2}\,dx, 
\end{align*} 
which proves \eqref{e8.2}. 
By \eqref{e8.1} and \eqref{e8.2}, we have  
$$\|I_\alpha(\eta)\|_2\leq C \|D_\alpha(\eta)\|_p + 
C\|I_\alpha(\eta)\|_p.  $$  
Thus if $\|D_\alpha(\eta)\|_p\leq C\|\eta\|_p$, 
we have 
$$\|I_\alpha(\eta)\|_2\leq C \|\eta\|_p + 
C\|I_\alpha(\eta)\|_p.  $$  
Using this with $\eta_t$ in place of $\eta$  
and homogeneity, we readily see that 
$$t^{\alpha-\gamma/2}\leq Ct^{-\gamma+\gamma/p} +Ct^{\alpha+\gamma(1/p-1)}
 \leq Ct^{-\gamma+\gamma/p} $$ 
for all $t \in (0,1)$, which implies that $p\geq 2\gamma/(\gamma+2\alpha)$ 
as claimed. 

\end{remark} 

\begin{remark} 
Define the Littlewood-Paley function  
$$g_Q(f)(x)=\left(\int_0^\infty |Q_t*f(x)|^2\, \frac{dt}{t}\right)^{1/2}, $$
where $Q$ is as in \eqref{ker2}.   Then it is known that 
$$c_1\|f\|_p\leq \|g_Q(f)\|_p\leq c_2\|f\|_p,\quad 1<p<\infty,  $$
with positive constants $c_1$, $c_2$ independent of $f$ (see \cite{R}).  
Also, we can show that 
\begin{equation}\label{e8.3}
g_Q(f)(x)\leq C_\alpha D_\alpha(f)(x), \quad 0<\alpha<1,  
\end{equation} 
for $f \in \mathscr S(\Bbb R^n)$, 
similarly to \cite[p. 162, 6.12]{St2}, which implies the reverse inequality of 
$\|D_\alpha(f)\|_p\leq C\|f\|_p$ in part (1) of Theorem $\ref{T1.1}$.  
\par 
Here we give a proof of \eqref{e8.3} in more details for completeness.   
Let 
$$U_\alpha(x,t)=K_t*I_\alpha(f)(x)=\int \hat{f}(\xi)(2\pi\rho(\xi))^{-\alpha} 
e^{-2\pi t\rho(\xi)} e^{2\pi i\langle x, \xi\rangle}\, d\xi,  $$   
where $K$ is as in \eqref{parapo}.  
Then 
$$\partial_0^2 U_\alpha(x,t)=\int \hat{f}(\xi)(2\pi\rho(\xi))^{-\alpha+2} 
e^{-2\pi t\rho(\xi)} e^{2\pi i\langle x, \xi\rangle}\, d\xi, $$  
where $\partial_0=\partial/\partial t$, and 
\begin{align*} 
\int_0^\infty \partial_0^2 U_\alpha(x,t+s)s^{-\alpha}\, ds 
&= \left(\int_0^\infty 
e^{-s}s^{-\alpha}\, ds\right) \int \hat{f}(\xi)(2\pi\rho(\xi)) 
e^{-2\pi t\rho(\xi)} e^{2\pi i\langle x, \xi\rangle}\, d\xi  
\\ 
&=-\Gamma(1-\alpha) \frac{1}{t}Q_t*f(x).  
\end{align*} 
Using this, we see that 
\begin{align*} 
&\left(\int_0^\infty |Q_t*f(x)|^2\, \frac{dt}{t}\right)^{1/2}
= \Gamma(1-\alpha)^{-1}\left(\int_0^\infty 
t\left|\int_0^\infty \partial_0^2 U_\alpha(x,t+s)s^{-\alpha}\, ds\right|^2  
\, dt \right)^{1/2} 
\\   
&=\Gamma(1-\alpha)^{-1}\left(\int_0^\infty 
\left|\int_0^\infty t^{3/2-\alpha} \chi_{[1,\infty)}\left(s\right)
|s-1|^{-\alpha} \partial_0^2 U_\alpha(x,st)\, ds\right|^2  
\, dt \right)^{1/2}. 
\end{align*}  
By Minkowski's inequality, this is bounded by 
\begin{align*} 
& \Gamma(1-\alpha)^{-1}\int_1^\infty (s-1)^{-\alpha}\left(\int_0^\infty 
t^{2(3/2-\alpha)} 
 \left|\partial_0^2 U_\alpha(x,st)\right|^2  \, dt \right)^{1/2} \, ds 
\\
&= \Gamma(1-\alpha)^{-1}\left(
\int_1^\infty (s-1)^{-\alpha}s^{-2+\alpha}\, ds\right) 
\left(\int_0^\infty 
t^{3-2\alpha} 
 \left|\partial_0^2 U_\alpha(x,t)\right|^2  \, dt \right)^{1/2}.    
\end{align*}
Thus 
\begin{equation}\label{e8.4}
 g_Q(f)(x)\leq C_\alpha \left(\int_0^\infty 
t^{3-2\alpha} 
 \left|\partial_0^2 U_\alpha(x,t)\right|^2  \, dt \right)^{1/2}.     
 \end{equation} 
\par 
Since $\int \partial_0^2K_t =0$, we have 
\begin{equation*} 
\partial_0^2U_\alpha(x,t)= \int \partial_0^2 K_t(y)I_\alpha f(x+y)\, dy  
= \int \partial_0^2 K_t(y)\left(I_\alpha f(x+y)-I_\alpha f(x)\right)\, dy.    
\end{equation*}
Arguing similarly to the proof of Lemma \ref{L3.1}, we see that 
$$|\partial_0^2 K_t(y)|\leq C(t+\rho(y))^{-\gamma-2}.   $$    
Using this, we have 
\begin{align*} 
&|\partial_0^2U_\alpha(x,t)|\leq  C\int (t+\rho(y))^{-\gamma-2} 
\left|I_\alpha f(x+y)-I_\alpha f(x)\right|\, dy 
\\ 
&\leq C\int\limits_{\rho(y)<t}  t^{-\gamma-2}
\left|I_\alpha f(x+z)-I_\alpha f(x)\right|\, dy   
+ C\int\limits_{\rho(y)\geq t} \rho(y)^{-\gamma-2}
\left|I_\alpha f(x+z)-I_\alpha f(x)\right|\, dy.         
\end{align*}@@
It follows that 
$$\int_0^\infty t^{3-2\alpha} 
 \left|\partial_0^2 U_\alpha(x,t)\right|^2  \, dt \leq C(I + II), $$ 
 where 
 $$ I=  \int_0^\infty t^{3-2\alpha} \left(\int_{\rho(y)<t}  t^{-\gamma-2}
\left|I_\alpha f(x+y)-I_\alpha f(x)\right|\, dy\right)^2 \, dt, $$
 $$ 
 II= \int_0^\infty t^{3-2\alpha}\left(\int_{\rho(y)\geq t} \rho(y)^{-\gamma-2}
\left|I_\alpha f(x+y)-I_\alpha f(x)\right|\, dy\right)^2 \, dt. $$
By the Schwarz inequality 
\begin{align*}  
I&\leq C\int_0^\infty t^{3-2\alpha}t^{-2(\gamma+2)}t^\gamma \int_{\rho(y)<t}  
\left|I_\alpha f(x+y)-I_\alpha f(x)\right|^2\, dy \, dt 
\\ 
&=C\int  
\left|I_\alpha f(x+y)-I_\alpha f(x)\right|^2\left(\int_{\rho(y)}^\infty 
t^{-1-\gamma-2\alpha} \, dt \right) dy
\\ 
&=C\frac{1}{\gamma+2\alpha} D_\alpha(f)(x)^2. 
\end{align*}  
Also,   
\begin{align*}  
II&\leq C \int_0^\infty t^{3-2\alpha} t^{-2}
\left(\int_{\rho(y)\geq t} \rho(y)^{-\gamma-2}
\left|I_\alpha f(x+y)-I_\alpha f(x)\right|^2\, dy\right) \, dt 
\\ 
&= C \int \rho(y)^{-\gamma-2}
\left|I_\alpha f(x+y)-I_\alpha f(x)\right|^2
\left(\int_0^{\rho(y)} t^{1-2\alpha} \, dt \right)  \, dy
\\ 
&=C\frac{1}{2-2\alpha} D_\alpha(f)(x)^2.  
\end{align*}
Therefore 
$$\left(\int_0^\infty t^{3-2\alpha} 
 \left|\partial_0^2 U_\alpha(x,t)\right|^2  \, dt\right)^{1/2}  
 \leq C_\alpha D_\alpha(f)(x). $$   
Combining this with \eqref{e8.4}, we have \eqref{e8.3}.

\end{remark}

\end{document}